\documentclass[12pt]{article}
\usepackage{e-jc}
\usepackage{amsmath,amsfonts}
\usepackage{epsfig}
\usepackage{multirow}
\newtheorem{theorem}{Theorem}
\newtheorem{proposition}[theorem]{Proposition}
\newtheorem{definition}[theorem]{Definition}
\newtheorem{lemma}[theorem]{Lemma}

\newenvironment{method}{\noindent\textbf{Method for computing edit distance:} \\}{}

 \font\xviiroman=cmr17
\def\udot{\mathbin{\ooalign{$\cup$\crcr
   \hfil\raise 8pt\hbox{\xviiroman.}\hfil\crcr}}}
\def\bigudotx#1#2{\mathop{\smash{\ooalign{$#1\bigcup$\crcr
   \hfil\raise 8pt\hbox{#2}\hfil\crcr}}\vphantom{\bigcup}}}

\newcommand{\dotcup}{\udot}

\newcommand{\F}{{\cal F}}
\newcommand{\K}{{\cal K}}
\newcommand{\PP}{{\cal P}}
\newcommand{\MM}{{\cal M}}
\newcommand{\HH}{{\cal H}}
\newcommand{\dist}{{\rm Dist}}
\newcommand{\forb}{{\rm Forb}}
\newcommand{\cmax}{c_{\rm max}}
\newcommand{\cmin}{c_{\rm min}}
\newcommand{\chib}{\chi_B}

\newcommand{\CP}{{\cal P}}

\newcommand{\vw}{{\rm VW}}
\newcommand{\vb}{{\rm VB}}
\newcommand{\ew}{{\rm EW}}
\newcommand{\eb}{{\rm EB}}
\newcommand{\eg}{{\rm EG}}

\newcommand{\vwk}{{\rm VW}(K)}
\newcommand{\vbk}{{\rm VB}(K)}
\newcommand{\ewk}{{\rm EW}(K)}
\newcommand{\ebk}{{\rm EB}(K)}
\newcommand{\egk}{{\rm EG}(K)}
\newcommand{\vwks}{|{\rm VW}(K)|}
\newcommand{\vbks}{|{\rm VB}(K)|}
\newcommand{\ewks}{|{\rm EW}(K)|}
\newcommand{\ebks}{|{\rm EB}(K)|}

\newcommand{\wk}{{\bf W}_K}
\newcommand{\bk}{{\bf B}_K}
\newcommand{\mkp}{{\bf M}_K(p)}
\newcommand{\mkstarp}{{\bf M}_{K^*}(p)}

\newenvironment{proof}{\noindent{\bf Proof.\,}}{\hfill$\Box$}

\newcommand{\A}{{\bf A}}

\newcommand{\M}{{\bf M}}
\newcommand{\one}{{\bf 1}}
\newcommand{\zero}{{\bf 0}}
\newcommand{\uu}{{\bf u}}
\newcommand{\x}{{\bf x}}
\newcommand{\z}{{\bf z}}

\newcommand{\del}{{\bf d}}
\newcommand{\arrows}{\mapsto_c}

\newcommand{\kw}{K_{\rm W}}
\newcommand{\kb}{K_{\rm B}}
\newcommand{\kbc}{K_{{\rm B}\setminus {\rm C}}}

\newcommand{\upspace}{\vspace{-9pt}}
\newcommand{\itsp}{\upspace\item}
\newcommand{\firstspace}{\vspace{4pt}}

\newcommand{\setm}{-}
\title{Edit distance and its computation}
\author{J\'ozsef Balogh\thanks{Department of Mathematics, University
of Illinois at Urbana-Champaign, Urbana, IL, 61801,
\texttt{jobal@math.uiuc.edu}. This author's research supported in
part by NSF grant DMS-0600303, UIUC Campus Research Board \#07048
and by OTKA grants T034475 and T049398.} \and Ryan
Martin\thanks{Department of Mathematics, Iowa State University,
Ames, IA 50011, \texttt{rymartin@iastate.edu}. This author's
research supported in part by NSA grant H98230-05-1-0257.}}

\date{\dateline{Sep 26, 2007}{Jan 17, 2008}\\
\small Mathematics Subject Classification: 05C35, 05C80}
\begin{document}
\maketitle
\begin{abstract}
In this paper, we provide a method for determining the asymptotic
value of the maximum edit distance from a given hereditary property.
This method permits the edit distance to be computed without using
Szemer\'edi's Regularity Lemma directly.

Using this new method, we are able to compute the edit distance
from hereditary properties for which it was previously unknown.  For
some graphs $H$, the edit distance from $\forb(H)$ is computed,
where $\forb(H)$ is the class of graphs which contain no induced
copy of graph $H$.

Those graphs for which we determine the edit distance asymptotically
are $H=K_a+E_b$, an $a$-clique with $b$ isolated vertices, and
$H=K_{3,3}$, a complete bipartite graph.  We also provide a graph, the first such construction, for which the edit distance cannot be determined just by considering partitions of the vertex set into cliques and cocliques.

In the process, we develop weighted generalizations of Tur\'an's
theorem, which may be of independent interest.
\end{abstract}


\section{Introduction}
Throughout this paper, we use standard terminology in the theory of
graphs.  See, for example,~\cite{Bmgt}.  A subgraph devoid of edges,
usually called an independent set, is referred to in this paper as a
\textbf{coclique}, so that it parallels the notion of a
\textbf{clique}.

\subsection{Background}
The edit distance of graphs was defined in~\cite{AKM}
as follows:
\begin{definition}
Let $\PP$ denote a class of graphs.  If $G$ is a fixed graph, then
\textbf{the edit distance from $G$ to $\PP$} is
$$ \dist(G,\PP)=\min\left\{\left|E(F)\triangle E(G)\right| :
   F\in\PP, V(F)=V(G)\right\} $$
and \textbf{the edit distance from $n$-vertex graphs to $\PP$} is
$$ \dist(n,\PP)=\max\left\{\dist(G,\PP) : |V(G)|=n\right\} . $$
\end{definition}

It is natural to consider hereditary properties of graphs.  A
\textbf{hereditary property} is one that is closed under the
deletion of vertices. In fact, edge-modification for such properties
is an important question in computer science, as described in Alon
and Stav~\cite{AS1} and biology, as shown in~\cite{AKM}.

Clearly, $\forb(H)$ is a hereditary property for any graph $H$.  In
fact, every hereditary property, $\HH$, can be expressed as
$\bigcap_{H\in\F(\HH)}\forb(H)$, where the intersection is over the
family $\F(\HH)$, which consists of all graphs $H$ which are the
minimal elements of $\overline{\HH}$.

In~\cite{AS1}, Alon and Stav prove that, for every hereditary
property $\HH$, there exists a $p^*=p^*(\HH)$ such that, with high
probability, $\dist(n,\HH)=\dist\left(G(n,p^*),\HH\right)+o(n^2)$,
where $G(n,p)$ denotes the usual Erd\H{o}s-R\'enyi random graph.
This fact can be used to prove the existence of
$$ d^*(\HH)\stackrel{\rm def}{=}
   \lim_{n\rightarrow\infty}\dist(n,\HH)/\binom{n}{2} .  $$

\subsection{Previous results}
The previously-known general bounds for $\dist(n,\forb(H))$ are
expressed in terms of the so-called binary chromatic number:
\begin{definition}
The \textbf{binary chromatic number} of a graph $G$, $\chib(G)$ is
the least integer $k+1$ such that, for all $c\in\{0,\ldots,k+1\}$,
there exists a partition of $V(G)$ into $c$ cliques and $k+1-c$
cocliques.
\end{definition}

The binary chromatic number~\cite{AKM} is called the ``colouring
number'' of a hereditary property by Bollob\'as and
Thomason~\cite{BT97} and again by Bollob\'as~\cite{Bher} and is
called the parameter $\tau(H)$ in Pr\"omel and Steger~\cite{PS3}.
The term indicates its generalizibility to multicolorings of the
edges of $K_n$, or $K_{n,n}$ as in~\cite{AM}.


The binary chromatic number gives the value of $\dist(n,\forb(H))$
to within a multiplicative factor of $2$, asymptotically:
\begin{theorem}[\cite{AKM}]
If $H$ is a graph with binary chromatic number $\chib(H)=k+1$, then
$\left(\frac{1}{2k}-o(1)\right)\binom{n}{2}
\leq\dist\left(n,\forb(H)\right)\leq\frac{1}{k}\binom{n}{2}$.
\label{thm:AKM}
\end{theorem}

\subsubsection{Known values of $d^*$ and $p^*$}
In~\cite{AKM}, a large class of graphs $H$ for which $d^*(\forb(H))$
is known to be the lower bound in Theorem~\ref{thm:AKM} is
described.  Namely, if a graph $H$ has the property that
$\chib(H)=k+1$ and there exist $(A,c)$ and $(a,C)$ such that each of
the following occurs
\begin{itemize}\firstspace
   \itsp $V(H)$ cannot be partitioned into $c$ cliques and $A$
   cocliques,
   \itsp $V(H)$ cannot be partitioned into $C$ cliques and $a$
   cocliques,
   \itsp $A+c=a+C=k$ and $c\leq k/2\leq C$,
\end{itemize}
then
$$ d^*(\forb(H))=\frac{1}{2k} . $$

It is observed in~\cite{AS1} that if $H$ and $(A,c)$ and $(a,C)$
satisfy the conditions above, then $p^*(\forb(H))=1/2$. Furthermore,
if $H$ is a self-complementary graph, then $A=C$ and $a=c$.  So,
$C+c=k$, which implies $c\leq k/2\leq C$ and
$\left(p^*(\forb(H)),d^*(\forb(H))\right)=\left(1/2,1/(2k)\right)$.

The edit distance from monotone properties is also well-known.  A
\textbf{monotone property} is, without loss of generality, closed
under the removal of either vertices or edges. Let $\MM$ be a
monotone property of graphs.  The theorems of Erd\H{o}s and
Stone~\cite{ES46} and Erd\H{o}s and Simonovits~\cite{ES66} give that
$$ d^*(\MM)=1/r,\qquad\mbox{where } r=\min\{\chi(F)-1 : F\not\in\MM\}
\qquad\text{and}\qquad p^*(\MM)=1.$$

Alon and Stav, in~\cite{AS2}, prove that
$ d^*(\forb(K_{1,3}))=p^*(\forb(K_{1,3}))=1/3 . $
In this paper, we generalize this result to compute the pairs
$(p^*,d^*)$ for hereditary properties of the form $\forb(K_a+E_b)$
and $\forb\left(\overline{K_a+E_b}\right)$, where $K_a$ is a
complete graph on $a$ vertices, $E_b$ is an empty graph on $b$
vertices and the ``$+$'' denotes a disjoint union of graphs.  The
claw $K_{1,3}$ is $\overline{K_3+E_1}$.

In both~\cite{AKM} and in~\cite{AS2}, more precise results for
determining $\dist(n,\HH)$ are given for several families of
hereditary properties.  For this paper, we concern ourselves
exclusively with the first-order asymptotics.

Finally, in~\cite{AS2}, a formula is given for the asymptotic value
of the distance\linebreak[4] $\dist(G(n,1/2),\HH)$ for an arbitrary hereditary
property $\HH$.  It generalizes the result, stated in~\cite{AS1} and
implicit from arguments in~\cite{AKM}, that almost surely,
$\dist(G(n,1/2),\forb(H))=\frac{1}{2(\chib(\forb(H))-1)}
\binom{n}{2}-o(n^2)$.  In this paper, we will further generalize
this by determining an asymptotic expression for $\dist(G(n,p),\HH)$
for all $p\in [0,1]$.

\subsection{Colored homomorphisms}
Next we recall three definitions from~\cite{AS1} which are
convenient for us.
\begin{definition}
A \textbf{colored regularity graph (CRG)}, $K$, is a complete graph
for which the vertices are partitioned $V(K)=\vw(K)\dotcup\vb(K)$
and the edges are partitioned
$E(K)=\ew(K)\dotcup\eg(K)\dotcup\eb(K)$.  The sets $\vw$ and $\vb$
are the white and black vertices, respectively, and the sets $\ew$,
$\eg$ and $\eb$ are the white, gray and black edges, respectively.
\end{definition}

Bollob\'as and Thomason (\cite{BT95},\cite{BT00}) originate the use
of this structure to define so-called basic hereditary properties.
In particular, the paper~\cite{BT00} generalizes the enumeration of
graphs with a given property $\cal P$ to the problem of
computing the probability that $G(n,p)\in\PP$.  The problems are
equivalent if $p=1/2$.  The papers use many of the techniques that
are repeated or cited in the subsequent works on edit distance and use other nontrivial ideas.

\begin{definition}
Let $K$ be a CRG with $V(K)=\{v_1,\ldots,v_k\}$.  The graph property
$\CP_{K,n}$ consists of all graphs $J$ on $n$ vertices for which
there is an equipartition $\A=\{A_i : 1\leq i\leq k\}$ of the
vertices of $J$ satisfying the following conditions for $1\leq
i<j\leq k$:
\begin{itemize}\firstspace
   \itsp if $v_i\in\vw(K)$, then $A_i$ spans an empty graph in $J$,
   \itsp if $v_i\in\vb(K)$, then $A_i$ spans a complete graph in
   $J$,
   \itsp if $\{v_i,v_j\}\in\ew(K)$, then $(A_i,A_j)$ spans an empty
   bipartite graph in $J$,
   \itsp if $\{v_i,v_j\}\in\eb(K)$, then $(A_i,A_j)$ spans a
   complete bipartite graph in $J$,
   \itsp if $\{v_i,v_j\}\in\eg(K)$, then $(A_i,A_j)$ is
   unrestricted.
\end{itemize}
If all of the above holds, we say that the equipartition
\textbf{witnesses the membership of $J$ in $\CP_{K,n}$}.
\end{definition}

\begin{definition}
A \textbf{colored-homomorphism} from a (simple) graph $F$ to a CRG,
$K$, is a mapping $\varphi : V(F)\rightarrow V(K)$, which satisfies
the following:
\begin{enumerate}
   \itsp If $\{u,v\}\in E(F)$ then either
   $\varphi(u)=\varphi(v)=t\in\vb(K)$, or $\varphi(u)\neq\varphi(v)$
   and $\left\{\varphi(u),\varphi(v)\right\}\in\eb(K)\cup\eg(K)$.
   \itsp If $\{u,v\}\not\in E(F)$ then either
   $\varphi(u)=\varphi(v)=t\in\vw(K)$, or $\varphi(u)\neq\varphi(v)$
   and $\left\{\varphi(u),\varphi(v)\right\}\in\ew(K)\cup\eg(K)$.
\end{enumerate}
Moreover, a colored-homomorphism can be defined from a CRG, $K'$, to
another CRG, $K''$, that satisfies the following:
\begin{enumerate}
\setcounter{enumi}{-1}
   \itsp If $v\in\vb(K')$, then $\varphi(v)\in\vb(K'')$.  If
   $v\in\vw(K')$, then $\varphi(v)\in\vw(K'')$.
   \itsp If $(u,v)\in\eb(K')$ then either
   $\varphi(u)=\varphi(v)=t\in\vb(K'')$, or
   $\varphi(u)\neq\varphi(v)$ and
   $\left(\varphi(u),\varphi(v)\right)\in\eb(K'')\cup\eg(K'')$.
   \itsp If $(u,v)\in\ew(K')$ then either
   $\varphi(u)=\varphi(v)=t\in\vw(K'')$, or
   $\varphi(u)\neq\varphi(v)$ and
   $\left(\varphi(u),\varphi(v)\right)\in\ew(K'')\cup\eg(K'')$.
\end{enumerate}
Note that we can use the second definition to include the first, by
defining $V(F)=\vw(F)\dotcup\vb(F)$ in such a way as to make the
colored-homomorphism legal with respect to the edge set.
\end{definition}

\begin{definition}
A CRG, $K'$, is \textbf{induced} in another CRG, $K$, if there is a
colored-homomor-phism $\varphi : V(K')\rightarrow V(K)$ such that
\begin{itemize}\firstspace
   \itsp $\varphi$ is an injection and \upspace\item for any
   $u,v\in V(K')$ for which
   $\left\{\varphi(u),\varphi(v)\right\}\in\eg(K)$, then
   $\{u,v\}\in\eg(K')$.
\end{itemize}
\end{definition}

\begin{definition}
A CRG, $K$, is an \textbf{$\HH$-colored regularity graph
($\HH$-CRG)} for a hereditary property $\HH$ if, for every graph
$J\not\in\HH$, there is no colored-homomorphism from $J$ to $K$.

Denote $\K(\HH)$ to be the family of all CRGs $K$ such that for
every graph $J\not\in\HH$ there is no colored-homomorphism from $J$
to $K$.  If there is no colored-homomorphism from $J$ to $K$, then
this is denoted as $J\not\arrows K$.  If there is a
colored-homomorphism from $J$ to $K$, then this is denoted as
$J\arrows K$.
\end{definition}

Observe that if $\HH=\bigcap_{H\in\F(\HH)}\forb(H)$, then an
$\HH$-CRG, $K$, is one such that for all $H\in\F(\HH)$, there is no
colored-homomorphism from $H$ into $K$.

\subsection{Functions of colored regularity graphs}
\subsubsection{Binary chromatic number}
Previous edit distance results were expressed in terms of the
so-called \textbf{binary chromatic number}, which can be viewed as
an invariant on CRGs for which the edge set is gray.
\begin{definition}
Let $K(a,c)$ denote the CRG with $a$ white vertices, $c$ black
vertices and all edges gray.

The \textbf{binary chromatic number} of a hereditary property $\HH$,
denoted $\chib(\HH)$, is the least integer $k+1$ such that,
$K(a,c)\not\in\K(\HH)$ for all $a,c$ such that $a+c=k+1$. This
definition means that $\chib(\forb(H))=\chib(H)$ for any graph $H$.
\end{definition}
This quantity is too specific for our purposes.  We need to
introduce a function that accounts for nongray edges in CRGs.

\subsubsection{The function $f$}
Given a CRG, $K$, we define two functions.  If $K$
has $k$ vertices, with the usual notation for the edge sets and the
vertex sets, then let
$$ f_K(p)\stackrel{\rm def}{=}
   \frac{1}{k^2}\left[p\left(\vwks+2\ewks\right)
                      +(1-p)\left(\vbks+2\ebks\right)\right] . $$

The function that defines $f_K(p)$ was introduced in~\cite{AS1} and
corresponds to equipartitioning the vertex set of some $G$ which is
chosen according to the distribution $G(n,p)$ and mapping the parts
of the partition to the vertices of $K$.  So, $f$ represents the
expected proportion of edges that are changed under the rule that if an edge
is mapped to a white edge or its endvertices are mapped to the same
white vertex, then the edge is removed and if a nonedge is mapped to
a black edge or its endvertices are mapped to the same black vertex,
then the edge is added.

The function $f_K(p)$, as a function of $p$, is a line with a slope
in $[-1,1]$.

\subsubsection{The function $g$}
The function $g_K(p)$ is defined by a quadratic program. It
corresponds not necessarily to an equipartition, but a partition
with optimal sizes.

In order to define $g$, we first define some matrices: Let $\wk$
denote the adjacency matrix of the graph defined by the white edges,
along with the first $\vwks$ diagonal entries being $1$ (corresponding to the white vertices) and the other
diagonal entries being $0$.  Let $\bk$ denote the adjacency matrix of
the graph defined by the black edges along with the last $\vbks$
diagonal entries being $1$  (corresponding to the black vertices) and the other diagonal entries being $0$.  We
define the matrix $\mkp$ as follows:
$$ \mkp=p\wk+(1-p)\bk . $$
With this, we define $g_K(p)$:
\begin{equation}
g_K(p):=\left\{\begin{array}{rlcl}
               \min & \multicolumn{3}{l}{\uu^T\mkp\uu} \\
               {\rm s.t.} & \uu^T\one & = & 1 \\
                          & \uu & \geq & \zero .
               \end{array}\right. \label{QP}
\end{equation}
If an optimal solution $\uu'$ has zero entries, then
$g_K(p)=g_{K^*}(p)$ for the CRG, $K^*$, induced in $K$, whose
vertices correspond to the nonzero entries of $\uu'$. (Note that
$K^*$ may depend on $\uu'$.)

\begin{lemma}
For any CRG $K$, and any $p\in[0,1]$, there
exists a CRG $K^*$, where $K^*$ is defined as a CRG induced in $K$ by the vertices which correspond to nonzero entries of $\uu'$, such that
$g_K(p)=g_{K^*}(p)=\frac{1}{\one^T\M_{K^*}^{-1}(p)\one}$.
\label{lem:glemma}
\end{lemma}

We prove Lemma~\ref{lem:glemma} in Section~\ref{sec:glemma}.

\section{Results}
\subsection{General bounds}

Theorem~\ref{thm:fgthm} is our main theorem, relating the functions
$f$ and $g$.  For $p\in(0,1)$, the notation $G(n,p)$ is the random
variable that represents a graph on $n$ vertices chosen by a random
process in which each edge is present independently with probability
$p$.  For $m\geq 1$, $G(n,m)$ is the random variable that represents
a graph on $n$ vertices chosen uniformly at random from all $n$
vertex graphs with $\lfloor m\rfloor$ edges.

\begin{theorem}
For a hereditary property $\HH=\bigcap_{H\in\F(\HH)}\forb(H)$, let
$\K(\HH)$ denote all CRGs $K$ such that $H\not\arrows K$ for each
$H\in\F(\HH)$.  Then, $d^*(\HH)\stackrel{\rm
def}{=}\lim_{n\rightarrow\infty}\dist(n,\HH)/\binom{n}{2}$ exists.  Define
$$ f(p)\stackrel{\rm def}{=}\inf_{K\in\K(\HH)}f_K(p)\qquad\mbox{and}\qquad g(p)\stackrel{\rm def}{=}\inf_{K\in\K(\HH)}g_K(p). $$
Then it is the case that $f(p)=g(p)$ for all $p\in[0,1]$,
$$ d^*(\HH)=\max_{p\in [0,1]}f(p)=\max_{p\in [0,1]}g(p), $$
and $p^*(\HH)$ is the value of $p$ at which $f$ achieves its
maximum.  In addition, the function $f(p)=g(p)$ is concave.

Furthermore, for all $p\in (0,1)$,
$$ \max_{G:e(G)=p{\scriptstyle
   \binom{n}{2}}}\left\{\dist(G,\HH)\right\}=f(p)\binom{n}{2}+o(n^2) , $$
and for all $\epsilon>0$,
$\dist\left(G\left(n,p{\scriptstyle \binom{n}{2}}\right),\HH\right)\geq f(p)\binom{n}{2}-\epsilon n^2$,
 with probability approaching 1 as $n\rightarrow\infty$.
\label{thm:fgthm}
\end{theorem}

Of course, by definition,
$\dist(n,\forb(\HH))=d^*(\HH)\binom{n}{2}+o(n^2)$. \\

\noindent\textbf{Remark:} The main theorem of Alon and Stav~\cite{AS1} states,
informally, that there exists a $p^*=p^*(\HH)$ such that $\dist(n,\HH)=\dist(G(n,p^*),\HH)$. Here, we compute the first-order asymptotic of the edit distance and show that $f(p)\binom{n}{2}$ is asymptotically the maximum edit distance among all graphs of density $p$, and is achieved by the random graph $G(n,p{\scriptstyle \binom{n}{2}})$.  Informally, $\dist\left(G(n,p{\scriptstyle \binom{n}{2}}),\HH\right)=f(p)\binom{n}{2}+o(n^2)$ and in the proof, we show, that $\dist\left(G(n,p),\HH\right)=f(p)\binom{n}{2}+o(n^2)$ as well.

In addition, Theorem~\ref{thm:fgthm} has the advantage
that the edit distance can be computed, asymptotically, without
direct use of Szemer\'edi's Regularity Lemma.  As we see in
Theorems~\ref{thm:kaeb},~\ref{thm:k33},~\ref{thm:h9}
and~\ref{thm:p3k1}, the function $f(p)$ is very
useful in computing the values of $\left(p^*(\HH),d^*(\HH)\right)$.

The method for computing $(p^*,d^*)$ in this paper follows the same
pattern for every hereditary property. \\

\begin{method}
\indent\textbf{Upper bound:} Carefully choose CRGs,
$K',K''\in\K(\HH)$ (possibly $K'=K''$) and compute $\max_{p\in
[0,1]}\min\left\{g_{K'}(p),g_{K''}(p)\right\}$.  This maximum is an
upper bound for $d^*(\HH)$.

\indent\textbf{Lower bound:} Let $p^*$ be the value of $p$ at which
the function $\min\left\{g_{K'}(p),g_{K''}(p)\right\}$ achieves its
maximum. For any $K\in\K(\HH)$, we try to show that $f_K(p^*)$ is at least the
upper bound value.  If this is the case, then we have computed
$d^*(\HH)$; moreover, $p^*(\HH)$ is the $p^*$ provided above.  In
order to do this, we use a type of weighted Tur\'an
theorem.
\end{method}

\subsection{The edit distance of $K_a+E_b$}
We give a class of graphs in which neither the upper nor the lower
bounds given by the binary chromatic number hold.
\begin{theorem}
Let $a\geq 2$ and $b\geq 1$ be positive integers.  Let $H=K_a+E_b$,
the disjoint union of an $a$-clique and a $b$-coclique. Then,
$$d^*\left(\forb(K_a+E_b)\right)  =  \frac{1}{a+b-1} \quad\text{ and }\quad
p^*\left(\forb(K_a+E_b)\right)  =  \frac{a-1}{a+b-1},
$$
i.e., $\dist(n,\forb(K_a\cup
E_b))=\frac{1}{a+b-1}\binom{n}{2}-o(n^2)$. \label{thm:kaeb}
\end{theorem}

We note that $\chib(K_a+E_b)=\max\{a,b+1\}$ and so
Theorem~\ref{thm:kaeb} is an improvement over~\cite{AKM} in the case when $a\neq b+1$.  It is also an improvement over
Proposition~\ref{prop:UB}, which appears below, in the case when $b>1$ and $a>2$.   Alon and Stav~\cite{AS2} prove the case when $a=3$ and $b=1$, the complement of the ``claw,'' $K_{1,3}$.

\subsection{A few specific graphs}
In all known examples of hereditary properties $\HH$, the point at
which $\left(p^*(\HH),d^*(\HH)\right)$ occurs is either the
intersection of two curves $g_{K'}(p)$, $g_{K''}(p)$ or is the maximum
of a single curve $g_{K'}(p)$.  In either case, each CRG can be
chosen to be one with only gray edges.

We compute the edit distance of two hereditary properties that
demonstrate the complexity of both $p^*$ and $d^*$.

\subsubsection{The graph $K_{3,3}$}
The graph $K_{3,3}$ has $d^*$ and $p^*$ defined by the local maximum
of a single curve $g_{K'}(p)$.
\begin{theorem}
The complete bipartite graph $K_{3,3}$ satisfies
$$ p^*\left(\forb(K_{3,3})\right)  =  \sqrt{2}-1 \quad\text{and}\quad
d^*\left(\forb(K_{3,3})\right)  =  3-2\sqrt{2} .$$ Moreover, $p^*$
is the local maximum of $g_{K'}(p)$, where $K'$ consists of one
white vertex, two black vertices and all gray edges. \label{thm:k33}
\end{theorem}

It should be noted that neither $p^*$ nor $d^*$ could be determined
for this hereditary property by the intersection of a finite number
of $f$ curves, simply because such intersections would occur at
rational points.  So, a sequence of CRGs would be required.  By
using the $g$ curves, however, we need only to use a single
CRG.

\subsubsection{The graph $H_9$}
Here, the graph we construct is formed by taking $C_9^2$ and adding
a triangle.  That is, if the vertices are $\{0,1,2,3,4,5,6,7,8\}$,
then $i\sim j$ iff $i-j\in\{\pm 1,\pm 2\}\pmod{9}$ or both $i$ and $j$ are congruent to 0 modulo $3$. For notational
simplicity, we call this graph $H_9$. See Figure~\ref{fig:h9}.
\begin{center}
\begin{figure}[ht]
   \centerline{\epsfig{file=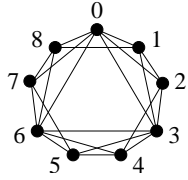,scale=0.5}}
   \caption{The graph $H_9$.} \label{fig:h9}
\end{figure}
\end{center}

An upper bound on $d^*(\forb(H_9))$ is defined by the intersection of
two curves, $g_{K'}(p)$, $g_{K''}(p)$, one of which corresponds to a
CRG that has one black edge.  For this graph $H_9$, it is impossible
to only consider CRGs which have all edges gray. It was a folklore
belief that for every graph it is sufficient to consider CRGs which
have all edges gray, $H_9$ is the first example showing that this
belief is false.
\begin{theorem}
The graph $H_9$ satisfies
\begin{equation*}
d^*\left(\forb(H_9)\right) \leq \frac{3-\sqrt{5}}{4} .
\end{equation*}
Moreover, this value occurs at the intersection of $g_{K'}(p)$ and
$g_{K''}(p)$, where $K'$ consists of two black vertices and a gray
edge and $K''$ consists of four white vertices, a black edge and 5
gray edges. \label{thm:h9}
\end{theorem}

In the proof, we show that if only gray-edge CRGs are used, then the
upper bound on $d^*$ could be no less than $1/5=0.2$, but
$\frac{3-\sqrt{5}}{4}\approx 0.191$.  The lower bound, from Theorem~\ref{thm:AKM}, is $d^*(\forb(H_9))\geq 1/6\approx 0.167$.

\subsection{4-vertex graphs}
In~\cite{AS2}, Alon and Stav compute
$\left(p^*(\forb(H)),d^*(\forb(H))\right)$ for all $H$ on at most 4
vertices.  Except for $P_3+K_1$ and its complement, all such graphs
$H$ are either covered by Theorem~\ref{thm:cnaught} (see
also~\cite{AKM}) or are of the form $K_a+E_b$ or
$\overline{K_a+E_b}$, which is covered by Theorem~\ref{thm:kaeb}. Here we give a short and different proof, using Lemma~\ref{lem:KaEb}, for $\overline{P_3+K_1}$, which consists of a triangle and a pendant edge.
\begin{theorem}
The graph $\overline{P_3+K_1}$ satisfies
$$
p^*\left(\forb(\overline{P_3+K_1})\right)=2/3\qquad\mbox{and}\qquad d^*\left(\forb(\overline{P_3+K_1})\right)=1/3 .
$$
\label{thm:p3k1}
\end{theorem}

\section{Basic tools}
\subsection{Improved binary chromatic number bounds}
Lemma~\ref{lem:glemma}, along with Theorem~\ref{thm:fgthm}, yields a
proof of a somewhat better upper bound for $\dist(n,\HH)$, based on the binary chromatic
number.

Recall that $K(a,c)$ denotes the CRG that consists of $a$ white
vertices, $c$ black vertices and only gray edges.  If $k=\chib(\HH)-1$, then let
$\cmin$ be the least $c$ so that $K(k-c,c)\not\in\K(\HH)$. Let
$\cmax$ be the greatest such number. For $\HH=\forb(H)$, there
exists an upper bound that can be expressed in terms of the binary
chromatic number of $H$ and corresponding $\cmin$ and $\cmax$.
\begin{theorem}[\cite{AKM}]
Let $H$ be a graph with binary chromatic number $k+1$ and $\cmin$
and $\cmax$ be defined as above.  If $\cmin\leq k/2\leq\cmax$, then
$$ d^*(\forb(H))=\frac{1}{2k} . $$
Otherwise, let $c_0$ be the one of $\{\cmax,\cmin\}$ that is closest
to $k/2$.  Then
$$ d^*(\forb(H))\leq
   \left(\frac{1}{1+2\sqrt{\frac{c_0}{k}
                           \left(1-\frac{c_0}{k}\right)}}\right)
   \frac{1}{k}\leq\frac{1}{k} . $$
\label{thm:cnaught}
\end{theorem}

Proposition~\ref{prop:UB} improves this general upper bound, not
only trivially by extending it to general hereditary properties, but
also by improving the case when $\cmax=0$ or $\cmin=k$.

\begin{proposition}
Let $\HH$ be a hereditary property with $k+1=\chib(\HH)$ and
$c_0,\cmax,\cmin$ defined analogously to Theorem~\ref{thm:cnaught}.
The bounds in Theorem~\ref{thm:cnaught} hold for $\HH$. Furthermore,
if $\HH\neq\forb(K_{k+1})$, then
$$ d^*(\HH)\leq\frac{1}{k+1} . $$
\label{prop:UB}
\end{proposition}

Note that $d^*(\forb(K_{k+1}))=\frac{1}{k}$ by Tur\'an's theorem. \\

\begin{proof}
If we restrict our attention to the CRGs in $\K(\HH)$ which are of the
form $K(a,c)$, then Theorem~\ref{thm:fgthm} gives that
$$ d^*(\HH)\leq\max_{p\in [0,1]}
   \inf_{K(a,c)\in\K(\HH)}\left\{g_{K(a,c)}(p)\right\}
   =\max_{p\in [0,1]}\min_{K(a,c)\in\K(\HH)}
    \left\{\frac{p(1-p)}{a(1-p)+cp}\right\} . $$

A word on why the ``$\inf$'' is made into a ``$\min$'':  Recall
that if $\HH=\bigcap_{H\in\F(\HH)}\forb(H)$, then $K\in\K(\HH)$
means that $H\not\arrows K$ for all $H\in\F(\HH)$.  Choose some
$H_0\in\F(\HH)$.  In order for $K\in\K(\HH)$, it must be the case
that $H_0\not\arrows K$.  But, there are only a finite number of
pairs $(a,c)$ such that $H_0\not\arrows K(a,c)$.  Indeed,
$H_0\arrows K(a,c)$ if either $a\geq\chi(H_0)$ or
$c\geq\chi(\overline{H_0})$.  Therefore, regardless of $\HH$, there
are only a finite number of $(a,c)$ for which $K(a,c)\in\K(\HH)$.

Suppose there exist different pairs $(a,C)$ and $(A,c)$ such that
$a+C=A+c=k$ and $c\leq k/2\leq C$.  We bound $d^*(\HH)$ by
$\max\limits_{p\in [0,1]}\min\left\{g_{K(a,C)}(p),
g_{K(A,c)}(p)\right\}$. If $p<1/2$, then
\begin{eqnarray*}
   C(1-2p) & > & c(1-2p) \\
   -C(1-p)+Cp & < & -c(1-p)+cp \\
   (k-C)(1-p)+Cp & < & (k-c)(1-p)+cp \\
   a(1-p)+Cp & < & A(1-p)+cp \\
   g_{K(a,C)}(p) & > & g_{K(A,c)}(p) .
\end{eqnarray*}
Similarly, if $p>1/2$, then $g_{K(a,C)}(p)<g_{K(A,c)}(p)$.  So,
$\max\limits_{p\in [0,1]}\min\left\{g_{K(a,C)}(p),
g_{K(A,c)}(p)\right\}$ occurs at the intersection of the two curves,
which is $\left(1/2,1/(2k)\right)$.

Otherwise, let $c_0$ be the value of $c$ for which $K(k-c,c)\in\K(\HH)$
that is the closest to $k/2$. Without loss of generality, assume that
$c_0=c<k/2$.  If $c_0>0$, then $k\geq 3$ and we may bound $d^*(\HH)$
by $g_{K(k-c_0,c_0)}(p)$, which achieves its maximum at
$p=\frac{\sqrt{k-c_0}}{\sqrt{k-c_0}+\sqrt{c_0}}$ and this maximum is
\begin{eqnarray*}
   \frac{1}{\left(\sqrt{k-c_0}+\sqrt{c_0}\right)^2}
   & = & \left(\frac{1}{1+2\sqrt{\frac{c_0}{k}
                          \left(1-\frac{c_0}{k}\right)}}\right)
         \frac{1}{k} \\
   & \leq & \left(\frac{1}{1+2\sqrt{\frac{1}{k}
                            \left(1-\frac{1}{k}\right)}}\right)
            \frac{1}{k}=\frac{1}{k+2\sqrt{k-1}}
   <\frac{1}{k+1} .
\end{eqnarray*}

Finally, consider the case when $c_0=0$.  If there exists some
$K(\alpha,\gamma)\in\K(\HH)$ with $\gamma\geq 1$, then
$g_{K(\alpha,\gamma)}(p)$ intersects $g_{K(k,0)}(p)$ at
$p=\frac{k-\alpha}{k-\alpha+\gamma}$ and the value of each function
at that point is
$\frac{k-\alpha}{k(k-\alpha+\gamma)}\leq\frac{1}{k+1}$, which has
equality only if $\gamma=1$ and $\alpha=0$.  If there is no such $K(\alpha,\gamma)$, then,
with $\HH=\bigcap_{H\in\F(\HH)}\forb(H)$, each $H\arrows K(0,1)$.
Therefore, each $H$ is a clique and so the smallest one defines
$\HH$.  The fact that $\chib(\HH)=k+1$ requires
$\HH=\forb(K_{k+1})$.

So, $d^*(\HH)\leq\frac{1}{k+1}$ unless $\HH=\forb(K_{k+1})$.
\end{proof}

\subsection{Proof of Lemma~\ref{lem:glemma}}
\label{sec:glemma} Let $\uu'$ be an optimal solution of (\ref{QP})
and, among such solutions, it is one with the most number of zero
entries. Let $K^*$ be a CRG that corresponds to $\uu'$. Let
$\mkstarp$ be the associated matrix and $\uu^*$ be the vector formed
by removing the zero entries from $\uu'$.  By assumption, $\uu^*$ is
an optimal solution of
\begin{equation}
g_{K^*(p)}:=\left\{\begin{array}{rlcl}
               \min & \multicolumn{3}{l}{\uu^T\mkstarp\uu} \\
               {\rm s.t.} & \uu^T\one & = & 1 \\
                          & \uu & \geq & \zero ,
               \end{array}\right. \label{QPstar}
\end{equation}
where from $K^*$ the vertices that correspond to the deleted $0$
coordinates of $\uu'$ and the corresponding rows and columns of
$\mkstarp$ are removed. Furthermore, all
entries of $\uu^*$ are strictly positive.

Suppose $\mkstarp$ is not invertible, with $\mkstarp\x=\zero$ where
$\x\neq\zero$ and $\x^T\one\neq 0$. Then rescale $\x$ so that
$\x^T\one=1$.  Choose an $\epsilon>0$ such that
$(1-\epsilon)\uu^*+\epsilon\x$ has all nonnegative entries. This is
possible and is a feasible solution to the quadratic program
(\ref{QPstar}), producing the value
$$ \left((1-\epsilon)\uu^*+\epsilon\x\right)^T
   \mkstarp\left((1-\epsilon)\uu^*+\epsilon\x\right)
   =(1-\epsilon)^2(\uu^*)^T\mkstarp\uu^* , $$
which contradicts the presumed optimal value.

Suppose $\mkstarp$ is not invertible, with $\mkstarp\x=\zero$ where
$\x\neq\zero$ and $\x^T\one=0$. Then rescale $\bf x$ so that
$\uu^*+\x$ has nonnegative entries and at least one zero entry. This
is a feasible solution to (\ref{QPstar}), producing the value
$$ \left(\uu^*+\x\right)^T\mkstarp
   \left(\uu^*+\x\right)=(\uu^*)^T\mkstarp\uu^* ,
   $$
but this contradicts the value of $\uu^*$ because this vector has
zero entries.  More zero entries can be appended to create a
solution of (\ref{QP}) which has more zeros than $\uu'$.

Therefore, we may assume that $\mkstarp$ is invertible.  Note that
both it and its inverse are symmetric matrices. Define the following
vector:
$\z:=\mkstarp^{-1}\one/\left(\one^T\mkstarp^{-1}\one\right)$. Choose
$\epsilon>0$ small enough so that both
$\frac{1}{1+\epsilon}\left(\uu^*+\epsilon\z\right)$ and
$\frac{1}{1-\epsilon}\left(\uu^*-\epsilon\z\right)$ have all entries
nonnegative.  Such an $\epsilon$ exists because all entries of
$\uu^*$ are positive.

These are each feasible solutions and when
$\frac{1}{1\pm\epsilon}\left(\uu^* \pm\epsilon\z\right)$ is placed in
(\ref{QPstar}), it gives the value
\begin{eqnarray*}
\lefteqn{\frac{1}{(1\pm\epsilon)^2}
   \left(\uu^*\pm\epsilon\z\right)^T
   \mkstarp\left(\uu^*\pm\epsilon\z\right)} \\
   & = & \frac{1}{(1\pm\epsilon)^2}\left[(\uu^*)^T\mkstarp\uu^*\pm 2\epsilon(\uu^*)^T\mkstarp\z
   +\epsilon^2\z^T\mkstarp\z\right] \\
   & = & \frac{1}{(1\pm\epsilon)^2}\left[(\uu^*)^T\mkstarp\uu^*
   \pm 2\epsilon\frac{(\uu^*)^T\mkstarp\mkstarp^{-1}\one}
   {\one^T\mkstarp^{-1}\one}\right. \\
   & & \hspace{1in}\left.
   +\epsilon^2\frac{\one^T\mkstarp^{-1}\mkstarp\mkstarp^{-1}\one}
    {\left(\one^T\mkstarp^{-1}\one\right)^2}\right] \\
   & = & \frac{1}{(1\pm\epsilon)^2}
   \left[(\uu^*)^T\mkstarp\uu^*
   +\frac{\epsilon^2\pm 2\epsilon}
         {\one^T\mkstarp^{-1}\one}\right] \\
   & = & (\uu^*)^T\mkstarp\uu^*+\epsilon(\pm 2+\epsilon)
   \left[\frac{1}{\one^T\mkstarp^{-1}\one}
   -(\uu^*)^T\mkstarp\uu^*\right] .
\end{eqnarray*}
If $(\uu^*)^T\mkstarp\uu^*\neq\left(\one^T\mkstarp^{-1}\one\right)^{-1}$, then either $\uu^*+\epsilon\z$ or $\uu^*-\epsilon\z$ is a better
solution to (\ref{QPstar}) than $\uu^*$, a contradiction.

So, (\ref{QPstar}), hence (\ref{QP}), has value $\left(\one^T\mkstarp^{-1}\one\right)^{-1}$. \hfill$\Box$

\section{Proof of Theorem~\ref{thm:fgthm}}
\label{sec:fgthm} Our proof has the following outline:
\renewcommand{\theenumi}{\Alph{enumi}}
\begin{enumerate}
   \itsp Show that every graph $G$ on $n$ vertices and
   $p\binom{n}{2}$ edges has $\dist(G,\HH)\leq f(p)\binom{n}{2}$.
   \label{it:step1}
   \itsp Show that $f$ is continuous and so it achieves its
   maximum. \label{it:step2}
   \itsp Show that, for any fixed $p$ and for $\epsilon$ small enough,
   $\dist\left(G(n,p),\HH\right)\geq
   f(p)\binom{n}{2}-2\epsilon n^2$ for $n$ sufficiently large. \label{it:step3}
   \itsp Show that $g(p)=f(p)$ for all $p$. \label{it:step4}
   \itsp Show that $g(p)$ is concave. \label{it:step5}
\end{enumerate}

\subsubsection*{\ref{it:step1}: Upper bound.}
Recall that $f(p)=\inf\limits_{K\in\K(\HH)}f_K(p)$ and
$g(p)=\inf\limits_{K\in\K(\HH)}g_K(p)$.

Let $G$ be an arbitrary graph on $n$ vertices with $p\binom{n}{2}$
edges.  Let $K\in\K(\HH)$ with $k=\vbks+\vwks$.  We will randomly
partition $V(G)$ into $k$ pieces and delete and add edges in a
manner determined by $K$. For each $v\in V(G)$, randomly, and
independently from other vertices, place $v$ into $V_i$ with
probability $1/k$. Moreover, label the vertices of $K$ with
$\{v_1,\ldots,v_k\}$. Create $G'$ from $G$ by performing the
following action for each distinct $i$ and $j$ in $[k]$:
\begin{itemize}\firstspace
   \itsp If $v_i\in\vwk$, then delete the edges in $G$ having both
   endpoints in $V_i$.
   \itsp If $v_i\in\vbk$, then add the non-edges in $G$ having both
   endpoints in $V_i$.
   \itsp If $\{v_i,v_j\}\in\ewk$, then delete the edges in $G$
   having one endpoint in $V_i$ and the other in $V_j$.
   \itsp If $\{v_i,v_j\}\in\ebk$, then add the edges in $G$ having one endpoint in $V_i$ and the other in $V_j$.
\end{itemize}
If there is an induced copy of $H$ in $G'$, then there is a
colored-homomorphism from $H$ to $K$.  Since $K\in\K(\HH)$, there is
no $H\in\F(\HH)$ for which $H\arrows K$.  Thus, $G'\in\HH$.

The probability that an edge is deleted is $(\vwks+2\ewks)/k^2$ and the probability that a nonedge is added is $(\vbks+2\ebks)/k^2$.  Therefore,  expected number of changes is
$$ p\binom{n}{2}\frac{\vwks+2\ewks}{k^2}
   +(1-p)\binom{n}{2}\frac{\vbks+2\ebks}{k^2}
   =f_K(p)\binom{n}{2} . $$

This implies that there is a partition which results in at most
$f_K(p)\binom{n}{2}$ changes in order to transform $G$ into some
$G'\in\HH$,  i.e., $\dist(G,\HH)/\binom{n}{2}\leq f_K(p)$. Since
this is true for any $K\in\K(\HH)$,
$\dist(G,\HH)/\binom{n}{2}\leq\inf_{K\in\K(\HH)}f_K(p)=f(p)$.

\subsubsection*{\ref{it:step2}: Continuity of $f$.}
We differ slightly from Alon and Stav in their approach
in~\cite{AS1} to ensure the continuity of $f$.  For terminology and
citations of the theorems below, see chapter 7 of
Rudin~\cite{Rudin}.

The set $\K(\HH)$ is countable since the set of finite CRGs is
countable. Therefore, we can linearly order the members of $\K(\HH)$
as $K_1,K_2,\ldots$.  Let $m_n(p)=\min_{i\leq n}f_{K_i}(p)$ and
$f(p)=\inf_i f_{K_i}(p)$. Since each $f_{K_i}(p)$ is a line with
slope in $[-1,1]$, each $m_n(p)$ is Lipschitz with coefficient $1$.
So, $\{m_n\}$ forms an equicontinuous, pointwise bounded family. As
such, $\{m_n\}$ has a uniformly convergent subsequence. The limit
must, therefore, be continuous.  Since $m_n\rightarrow f$ pointwise,
this limit is $f(p)$.

Since $f$ is continuous, it achieves its maximum in the closed
interval $[0,1]$. Therefore, $\dist(n,\HH)\leq\max_{p\in [0,1]}
f(p)$. Define $p^*$ so that $f(p^*)$ is this maximum. Note that if
some lines are horizontal then $p^*$ is not necessarily
unique.

\subsubsection*{\ref{it:step3}: Lower bound for the random graph.}
Fix $p\in (0,1)$ and $\epsilon>0$. Let $S=S(\epsilon,\HH)$ the
function provided by the generalization of the Regularity Lemma,
cited as Lemma 2.7 in~\cite{AS1}.  The proof below follows ideas
similar to those in~\cite{AS1}.  Let $G\sim G(n,p)$.  A routine
application of the Chernoff bound (see~\cite{JLR}) gives that the
probability that every equipartition of $V(G)$ into $k\leq S$ pieces
$V_1\dotcup\cdots\dotcup V_k$ has the subgraphs $G[V_i]$ and the
bipartite subgraphs $G[(V_i,V_j)]$ with density in
$(p-n^{-0.4},p+n^{-0.4})$ for all distinct $i,j\in [k]$ is at most
$\exp\{-\Omega(n^{1.2})\}$, with $p$ and $S$ fixed. Choose $n$ to be
large enough for such a graph to exist and choose $G$ to be one such
graph.

Let $G'\in\HH$ have the property that $\dist(G,G')=\dist(G,\HH)$.
Apply the generalization of the Regularity Lemma to $G'$, with
parameters $\epsilon$ and $m=2\epsilon^{-1}$.  There is an
$S=S(\epsilon,\HH)$ such that there is an equipartition of the
vertex set: $V(G')=V_1\dotcup\cdots\dotcup V_k$, with $m\leq k\leq
S$. Each piece is of size either $L\stackrel{\rm def}{=}\lfloor
n/k\rfloor$ or $\lceil n/k\rceil$.

The graph $G''$ is constructed from this partition in such a way as
to ensure that $G''[V_i]$ is either an empty or complete graph and
either $d_{G''}(V_i,V_j)=0$ or $d_{G''}(V_i,V_j)=1$ or
$\epsilon/2\leq d_{G''}(V_i,V_j)\leq 1-\epsilon/2$. This is done by
deleting edges from sparse clusters and pairs and adding edges to
dense clusters and pairs.  Consequently, $\dist(G',G'')<
(\epsilon/2)n^2$.

This naturally yields a CRG, $K$, on the vertex
set $\{v_1,\ldots,v_k\}$ where $v_i$ is
$\left\{\mbox{white},\mbox{black}\right\}$ iff $G''[V_i]$ is
$\left\{\mbox{empty},\mbox{complete}\right\}$ and $\{v_i,v_j\}$ is
$\left\{\mbox{white},\mbox{black}\right\}$ iff
$\left\{d_{G''}(V_i,V_j)=0\right.$,\linebreak[4] $\left. d_{G''}(V_i,V_j)=1\right\}$; otherwise
$\{v_i,v_j\}$ is gray.  If there is a colored-homomorphism from
$H\in\F(\HH)$ to $K$, then the construction of $G''$ ensures that
$H$ is induced in both $G''$ and $G'$.  Therefore, $K$ must be in
$\K(\HH)$.

Since the distance between graphs is simply a symmetric difference
of edges, we see immediately that the triangle inequality applies:
\begin{eqnarray*}
   \dist(G,G') & \geq & \dist(G,G'')-\dist(G',G'') \\
   & \geq & \dist(G,G'')-(\epsilon/2) n^2 \\
   & \geq & \left(p-n^{-0.4}\right)\binom{L}{2}\vwks
   +\left(1-p-n^{-0.4}\right)\binom{L}{2}\vbks \\
   & & +\left(p-n^{-0.4}\right)L^2\ewks
   +\left(1-p-n^{-0.4}\right)L^2\ebks  -(\epsilon/2) n^2 \\
   & \geq &
   \frac{1}{k^2}\left(p\vwks+(1-p)\vbks\right)\binom{n}{2}
   \frac{(n-k)(n-2k)}{n(n-1)} \\
   & & +\frac{1}{k^2}\left(p\ewks+(1-p)\ebks\right)\binom{n}{2}
   \frac{(n-k)^2}{n(n-1)} \\
   & & -\frac{n^{1.6}}{2k}-\frac{n^{1.6}}{2}-(\epsilon/2) n^2 .
\end{eqnarray*}

For $n$ large enough,
\begin{eqnarray*}
   \dist(G,G') & \geq & \frac{1}{k^2}\left(p\vwks+(1-p)\vbks\right)\binom{n}{2}
   \left(1-\frac{3k}{n}\right) \\
   & & +\frac{1}{k^2}\left(p\ewks+(1-p)\ebks\right)\binom{n}{2}
   \left(1-\frac{2k}{n}\right)-\frac{3\epsilon}{4} n^2 \\
   & \geq & f_K(p)\binom{n}{2}-\epsilon n^2 .
\end{eqnarray*}

So, for each sufficiently small $\epsilon>0$, the probability that $G\sim G(n,p)$ satisfies
$\dist(G,\HH)\geq f(p)\binom{n}{2}-\epsilon n^2$ approaches $1$ as
$n\rightarrow\infty$.

The only place where randomness is used above is to show
that, with respect to any equipartition with $k\leq S$ parts, the
density of the pairs is close to $p$.  This is true for
$G\left(n,p{\scriptstyle\binom{n}{2}}\right)$ as well, therefore  we
 conclude that for all $\epsilon$ sufficiently small, the
probability that $G\sim G\left(n,p{\scriptstyle\binom{n}{2}}\right)$
satisfies $\dist(G,\HH)\geq f(p)\binom{n}{2}-\epsilon n^2$
approaches $1$ as $n\rightarrow\infty$.

Thus, $f(p)$ is the supremum of $\dist(G,\HH)$ for graphs $G$ of
density $p$ and\linebreak[4] $\dist(n,\HH)/\binom{n}{2}=f(p^*)-o(1)$.

\subsubsection*{\ref{it:step4}: Equality of $f$ and $g$.}
We address the $g$ functions.  Recalling (\ref{QP}),
$$ g_K(p)=\left\{\begin{array}{rlcl}
                    \min & \multicolumn{3}{l}{{\bf w}^T\mkp{\bf w}} \\
                    {\rm s.t.} & {\bf w}^T{\bf 1} & = & 1 \\
                               & {\bf w} & \geq & {\bf 0} .
                 \end{array}\right. $$

If $K$ has $k$ vertices, then ${\bf w}=\frac{1}{k}{\bf 1}$ is a
feasible solution, and $g_K(p)\leq f_K(p)$ for all $p\in [0,1]$.
Thus, $g(p)\leq f(p)$.

Fix $p\in[0,1]$ and $\epsilon\in(0,1)$ and choose a $K^*\in\K(\HH)$
such that $g_{K^*}(p)\leq g(p)+\epsilon/2$ and an optimal
solution in the corresponding quadratic program,
$\uu^*=(u_1,\ldots,u_k)$, has strictly positive entries.  We will
find a CRG, $L$ (for which the $\ell$ clusters are equally
weighted), that will approximate the weighted version of $K^*$.  Set
$\ell>5k\epsilon^{-1}$. Construct a CRG, $L$, on $\ell$ vertices
such that there are $\lfloor u_i\ell\rfloor$ or $\lceil
u_i\ell\rceil$ copies of vertex $x_i$ of $K^*$ in the natural way:
Let $y'$ be a copy of $x_i$ and $y''$ be a copy of $x_j$.
The vertex $y'$ has the same color as $x_i$ and $y''$ has the same
color as $x_j$.  If $i\neq j$, then $\{y',y''\}$ has the same color
as $\{x_i,x_j\}$. If $i=j$, then $\{y',y''\}$ has the same color as
vertex $x_i$.

Let $\tilde{\uu}=\left(\lceil u_1\ell\rceil,\ldots,\lceil
u_k\ell\rceil\right)$ and $\del=\tilde{\uu}-\ell{\uu}^*$.  Hence,
coordinatewise, $\zero\leq\del\leq k\one$. We can upper bound the
$f$ function of $L$:

\begin{eqnarray*}
   f_L(p) & = & \frac{1}{\ell^2}(\tilde{\uu})^T\mkstarp
   \tilde\uu \\
   & = & \frac{1}{\ell^2}\left(\ell{\uu}^*+\del\right)^T
   \mkstarp\left(\ell{\uu}^*+\del\right) \\
   & = & (\uu^*)^T\mkstarp\uu^*
   +\frac{2}{\ell}\uu^*\mkstarp\del
   +\frac{1}{\ell^2}\del^T\mkstarp\del \\
   & \leq & g_{K^*}(p)+\frac{2}{\ell}(\uu^*)^T{\bf J}\one
   +\frac{1}{\ell^2}\one^T{\bf J}\one \\
   & = & g_{K^*}(p)+\frac{2k}{\ell}(\uu^*)^T\one
   +\frac{k}{\ell^2}\one^T\one \\
   & = & g_{K^*}(p)+\frac{2k}{\ell}+\frac{k^2}{\ell^2} ,
\end{eqnarray*}
where $\bf J$ is the all ones $k\times k$ matrix. Since
$k/\ell<\epsilon/5<1/5$, it is true that
$2k/\ell+k^2/\ell^2<\epsilon/2$.  Therefore,
$$ f(p)\leq f_L(p)<g_{K^*}(p)+\frac{\epsilon}{2}<g(p)+\epsilon , $$
for all $\epsilon\in(0,1)$, yielding   $f(p)=g(p)$.

\subsubsection*{\ref{it:step5}: Concavity of $f(p)$.}
A function $h$ is concave on an interval domain if, whenever $a$ and $b$ are in the domain of $h$, then $h(ta+(1-t)b)\geq th(a)+(1-t)h(b)$ for all $t\in [0,1]$.

For the function $f$, the infimum of linear functions,
\begin{eqnarray*}
   f(ta+(1-t)b) & = & \inf_{K\in\K(\HH)}\{f_K(ta+(1-t)b)\}
   =\inf_{K\in\K(\HH)}\left\{tf_K(a)+(1-t)f_K(b)\right\} \\
   & \geq & t\left(\inf_{K\in\K(\HH)}\{f_K(a)\}\right) +(1-t)\left(\inf_{K\in\K(\HH)}\{f_K(b)\}\right) \\
   & = & tf(a)+(1-t)f(b) .
\end{eqnarray*}

\noindent This concludes
the proof of Theorem~\ref{thm:fgthm}. \hfill$\Box$

\section{The computation of $p^*$ and $d^*$ for specific families}

Let $t(n,k)$ denote the number of edges in the Tur\'an graph on $n$
vertices with no clique of order $k+1$.  The following is a result
of elementary computation:

$$ \frac{k-1}{k}\frac{n^2}{2}-\frac{k}{8}\leq
   t(n,k)=\frac{k-1}{k}\frac{n^2}{2}
   -\frac{k}{2}\left(\left\lceil\frac{n}{k}\right\rceil-\frac{n}{k}\right)
               \left(\frac{n}{k}-\left\lfloor\frac{n}{k}\right\rfloor\right)
               \leq\frac{k-1}{k}\frac{n^2}{2} . $$

\subsection{General approach}

To prove upper bounds on $d^*$, we use (\ref{QP}) and choose CRGs whose
curves intersect at $(p^*,d^*)$ or a curve that achieves its maximum
at $(p^*,d^*)$.

To prove lower bounds on $d^*$, we need to use a weighted Tur\'an approach
which seems to be quite difficult in general.  To see a simple
application of the weighted Tur\'an method, we provide a very short
proof of the lower bound in Theorem~\ref{thm:AKM} below:

Let $H$ be a graph with binary chromatic number $\chib$ and let $K$
be any CRG for which $H\not\arrows K$.  This immediately implies
that $K$ contains no clique of order $\chib+1$ whose edges are all
gray.

In particular, this implies that $\egk\leq t(k,\chib)$.  Setting
$p=1/2$, we see that
\begin{eqnarray*}
   f_K(1/2) & = &
   \frac{1}{k^2}\left[\frac{1}{2}\left(\vwks+2\ewks\right)
   +\frac{1}{2}\left(\vbks+2\ebks\right)\right] \\
   & = & \frac{1}{k^2}\left[\frac{k}{2}
   +\left(\ewks+\ebks\right)\right] \\
   & \geq & \frac{1}{k^2}\left[\frac{k}{2}
   +\left(\binom{k}{2}-t(k,\chib)\right)\right] \\
   & \geq & \frac{1}{2}-\frac{1}{k^2}t(k,\chib) \\
   & \geq & \frac{1}{2}-\frac{\chib-1}{2\chib}=\frac{1}{2\chib} ,
\end{eqnarray*}
and this proves the lower bound of Theorem~\ref{thm:AKM}.

\subsection{Edit distance of $K_a+E_b$}

\subsubsection{Upper bound}
Here, we choose $K'$ to have $|{\rm VW}(K')|=a-1$, $|{\rm
VB}(K')|=0$ and all edges gray.  Furthermore, we choose $K''$ to
have $|{\rm VW}(K'')|=0$, $|{\rm VB}(K'')|=b$ and all edges gray. It
is easy to see that both $K_a+E_b\not\arrows K'$ and
$K_a+E_b\not\arrows K''$.  An easy computation gives that
$g_{K'}(p)=\frac{p}{a-1}$ and $g_{K''}(p)=\frac{1-p}{b}$. The
intersection of the two functions is at the point
$(p^*,d^*)=\left(\frac{a-1}{a+b-1},\frac{1}{a+b-1}\right)$. Moreover, the fact that $\min\{g_{K'}(p),g_{K''}(p)\}$ is strictly unimodal, means that our proof below that $g(p^*)\geq d^*$ means that $p^*$ is the unique value at which $g(p)$ achieves its maximum.

\subsubsection{Weighted Tur\'an lemma}
The following lemma can be considered to be a generalization of
Tur\'an's theorem.  That is, if from Lemma~\ref{lem:KaEb} we only
apply condition (1)
but not condition (2),
then the answer is a basic consequence of Tur\'an.
\begin{lemma}
   Let $a\geq 2$ and let $K$ be a CRG with the property that
   any set $A$ of $a$ vertices has at least one of the following
   conditions:
   \renewcommand{\labelenumi}{(\arabic{enumi})}
   \begin{enumerate}
      \vspace{-9pt}\item $A$ contains at least one white edge,
      \label{it:white}
      \vspace{-9pt}\item $A$ contains a spanning subgraph of black edges.
      \label{it:black}
   \end{enumerate}
   Then
   $$ (a-1)\ewk+\ebk\geq\left\lceil\frac{n}{2}(n-a+1)\right\rceil
   . $$
   \label{lem:KaEb}
\end{lemma}

\begin{proof}
   We fix an integer $a\geq 2$ and proceed via induction on $n$.
   The base case, $n\leq a$ is trivial.

   Now, we assume that any CRG, $K'$, on $s<n$
   vertices that satisfies the conditions of the lemma has
   $(a-1)\ew(K')+\eb(K')\geq\lceil (s/2)(s-a)\rceil$.

   Let $K$ be a CRG on $n$ vertices.
   If it consists of only white edges, then
   $$ (a-1)\ewk+\ebk=(a-1)\ewk=(a-1)\binom{n}{2}
      \geq\left\lceil\frac{n}{2}(n-a+1)\right\rceil . $$

 Let   $V(K)\setm S$ be a maximal set of vertices that does not span a white edge.
  We may assume that $S\neq\emptyset$ because otherwise the minimum
   black degree is at least $n-a+1$, proving the claim of the
   theorem.  By the maximality of $V(K)\setm S$, for any $s\in S$ there exists a
   $t\in V(K)\setm S$ such that $st\in \ewk$.  Moreover, since
   there is no white edge in $V(K)\setm S$,
   vertex $s$ has at most $a-2$ gray neighbors in $V(K)\setm
   S$.  Otherwise, $s$ and $a-1$ gray neighbors in $V(K)\setm
   S$ will violate both conditions.

   The total weight of $K$ is as follows:
   \begin{itemize}\firstspace
      \itsp In the CRG induced by the vertex subset
      $V(K)\setm S$, the  weight is at least\linebreak[4]
      $\left\lceil (n-|S|)(n-|S|-a+1)/2\right\rceil$, by induction.
      \itsp In the CRG induced by the pair $(S,V(K)\setm S)$, each $s\in S$
       has at least one white neighbor and at most $a-2$ gray neighbors,
        so the  weight from $s$ into $V(K)\setm S$ is at least $(a-1)+(n-|S|-(a-2)-1)=n-|S|$.
      So the weight is at least $|S|(n-|S|)$.
      \itsp In the CRG induced by $S$, the total weight is at
      least $\left\lceil (|S|/2)(|S|-a+1)\right\rceil$, by
      induction.
   \end{itemize}
   Adding these together, the proof is complete. \\
\end{proof}

\noindent\textbf{Remark:} Note that equality holds when $S=\emptyset$ (i.e, there is no white edge)
 and the gray edges form a graph
that is either $(a-2)$-regular or has $n-1$ vertices of degree $a-2$
and one vertex of degree $a-3$, depending on divisibility.

\subsubsection{Lower bound}
Fix $p^*=\frac{a-1}{a+b-1}$.  Let $K$ be any CRG for which
$K_a+E_b\not\arrows K$.  To simplify notation, define $\kw$ to
be the CRG induced by $\vwk$.  First we will give a lower bound on
$p^*\ewks+(1-p^*)\ebks$.
\begin{itemize}\firstspace
   \itsp In the bipartite CRG induced by $(\vwk,\vbk)$, all edges must be
   black, otherwise $K_a+E_b\arrows K$.  These edges contribute
   a weight of $(1-p^*)|\vwk||\vbk|$.
   \itsp In the CRG induced by $\vbk$, each set of $b+1$ vertices
   has at least one black edge in the CRG they induce, otherwise
   $K_a+E_b\arrows K$.   The $a$-clique maps to one vertex and the
   $b$-coclique maps to the remaining $b$ black vertices.  By
   Tur\'an's theorem, these edges contribute a weight
   of
   at least $(1-p^*)\left[\binom{\vbks}{2}-t(\vbks,b)\right]$.
   \itsp In the CRG induced by $\vwk$, consider a set of $a$
   vertices.  If there is neither a white edge nor a spanning
   subgraph of black
   edges, then the vertices can be labeled $v_1,\ldots,v_a$ such
   that the $a-1$ edges incident to $v_1$ are all gray and
   $\{v_2,\ldots,v_a\}$ induces a CRG with all edges either gray
   or black.\\
   In this case, map the $b$-coclique to $v_1$ and the $a$ vertices
   of the clique to $v_1,\ldots,v_a$.  This exhibits the fact that
   $K_a+E_b\arrows K$.  We will apply Lemma~\ref{lem:KaEb}
   to $\kw$.  As a result, these edges contribute
   a weight of at least
   \begin{eqnarray*}
      \lefteqn{p^*|{\rm EW}(\kw)|+(1-p^*)|{\rm EB}(\kw)|} \\
      & = & \frac{1}{a+b-1}\left[(a-1)|{\rm EW}(\kw)|
      +b|{\rm EB}(\kw)|\right] \\
      & \geq & \frac{1}{a+b-1}\left[(a-1)|{\rm EW}(\kw)|
      +|{\rm EB}(\kw)|\right] \\
      & \geq &
      \frac{1}{a+b-1}\left\lceil\frac{\vwks}{2}
      (\vwks-a+1)\right\rceil .
   \end{eqnarray*}
\end{itemize}
The remaining edges of the CRG contribute the following to the weight:
\begin{eqnarray*}
   \lefteqn{(1-p^*)  \vwks\vbks
   +(1-p^*)\left(\binom{\vbks}{2}-t\left(\vbks,b\right)\right)} \\
   & \geq & \frac{1}{a+b-1}\left(b\vwks\vbks
   +b\left(\binom{\vbks}{2}-t\left(\vbks,b\right)\right)\right) \\
   & \geq & \frac{1}{a+b-1}\left(b\vwks\vbks
   +\frac{\vbks^2}{2}-\frac{b\vbks}{2}\right). \\
\end{eqnarray*}
Computing $f_K(p^*)$ gives, by definition,
\begin{eqnarray*}
   f_K(p^*) & = &
   \frac{1}{k^2}\left(p^*\left(\vwks+2\ewks\right)
   +(1-p^*)\left(\vbks+2\ebks\right)\right) \\
   & \geq & \frac{1}{(a+b-1)k^2}\left((a-1)\vwks+b\vbks
   +2b\vwks\vbks\right. \\
   & & \left.+\vbks^2-b\vbks+\vwks(\vwks-a+1)\right) \\
   & = & \frac{1}{(a+b-1)k^2}\left(2b\vwks\vbks
   +\vbks^2+\vwks^2\right) \\
   & = & \frac{1}{(a+b-1)k^2}\left(k^2+2(b-1)\vwks\vbks\right) \\
   & \geq & \frac{1}{a+b-1} .
\end{eqnarray*}
Therefore, $d^*(\forb(K_a+E_b))=\frac{1}{a+b-1}$ and
$p^*(\forb(K_a+E_b))=\frac{a-1}{a+b-1}$. \hfill$\Box$

\subsection{Edit distance of $K_{3,3}$}

\subsubsection{Upper bound}
The Young tableau in Figure~\ref{fig:k33young} diagrams the values
of $(a,c)$ for which $K_{3,3}\not\arrows K(a,c)$ and
Figure~\ref{fig:k33graph} gives the graph of $K(1,2)$ with the
region it defines shaded.
\begin{center}
\begin{figure}[ht]
   \hfill
   \begin{minipage}[t]{125pt}
      \begin{center}
         \epsfig{file=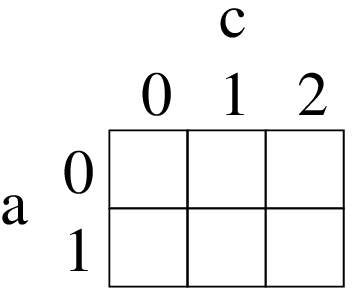,scale=0.5}
         \caption{The Young tableau of $(a,c)$ for which
         $K_{3,3}\not\arrows K(a,c)$.}
         \label{fig:k33young}
      \end{center}
   \end{minipage}
   \hfill
   \begin{minipage}[t]{250pt}
      \begin{center}
         \epsfig{file=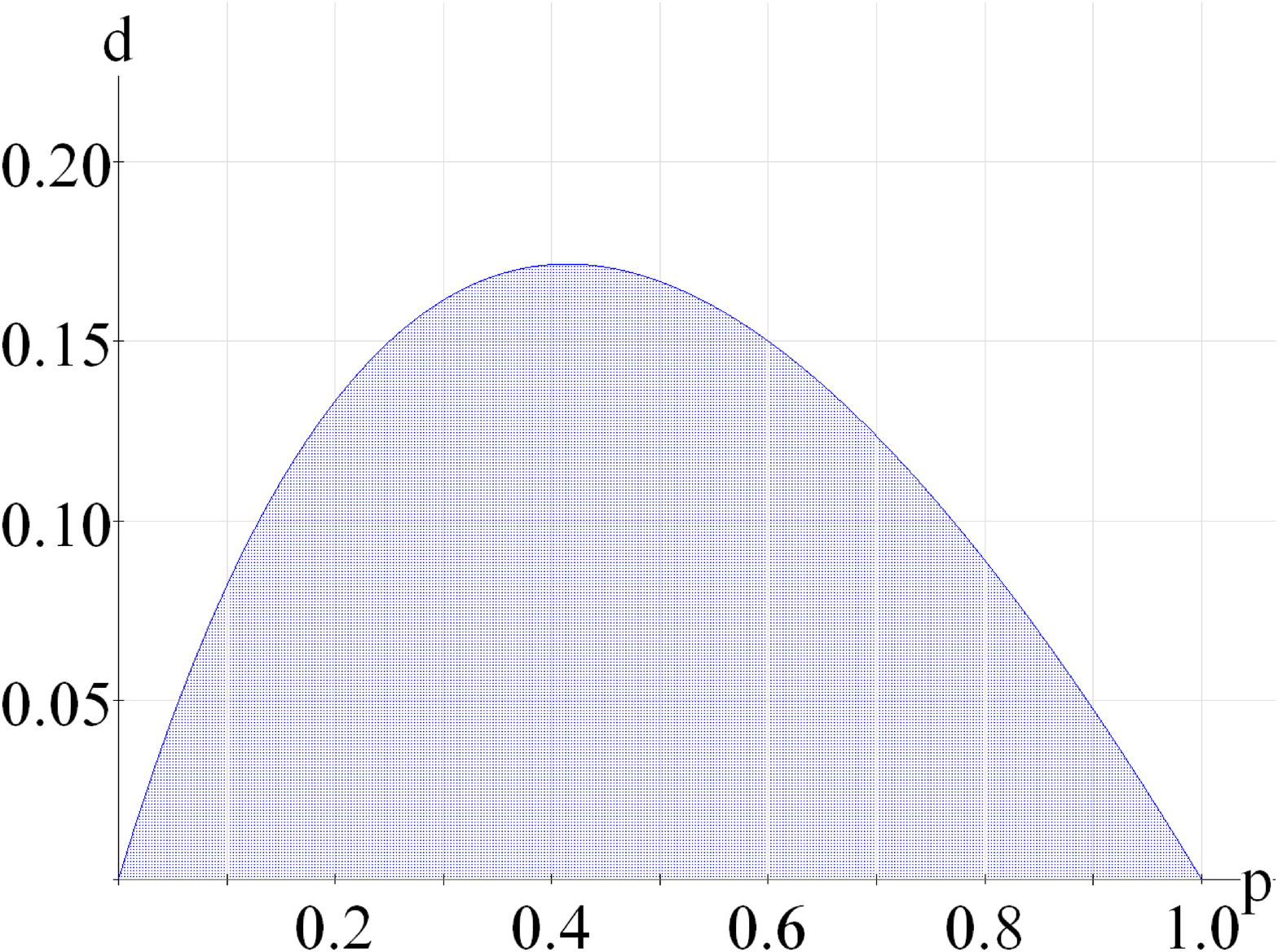,scale=0.15}
         \caption{The graph of $g_{K(1,2)}(p)$.}
         \label{fig:k33graph}
      \end{center}
   \end{minipage}
\end{figure}
\end{center}

Here, we choose $K'$ to have $|{\rm VW}(K')|=1$, $|{\rm VB}(K')|=2$
and all edges are gray.  That is, $K'=K(1,2)$ and it is easy to see that
$K_{3,3}\not\arrows K'$ .  We can use Lemma~\ref{lem:glemma} to
compute that $g_{K'}(p)=\frac{p(1-p)}{1+p}$. The maximum of this
function on $[0,1]$ occurs at
$(p^*,d^*)=\left(\sqrt{2}-1,3-2\sqrt{2}\right)$.

\subsubsection{Lower bound}
Fix $p^*=\sqrt{2}-1$.  Let $K$ be any CRG for which
$K_{3,3}\not\arrows K$.  For simplicity of notation, define $\kb$ to
be the CRG induced by $\vbk$.  First we will give a lower bound on
$p^*\ewks+(1-p^*)\ebks$.
\begin{itemize}\firstspace
   \itsp In the CRG induced by $\vwk$, all edges must be white,
   otherwise $K_{3,3}\arrows K$.  These edges contribute a weight of
   $p^*\binom{\vwks}{2}$.
   \itsp In the bipartite CRG induced by $\left(\vwk,\vbk\right)$, if there
   is a triangle $\{b_1,b_2,b_3\}$ in $\vbk$ that has all edges
   white or gray then, for every $w\in\vwk$, $\{b_i,w\}$ is white
   for at least one $i\in\{1,2,3\}$.  Otherwise, $K_{3,3}\arrows
   K$.\\
   Since there must be a white edge between every white/gray
 triangle in $\vbk$ and every vertex in $\vwk$, let $C\subseteq\vbk$ be a minimum-sized vertex set that
   contains a vertex from every triangle with no black edges in
   $\kb$.  These edges contribute a weight of at least
   $p^*\vwks|C|$.
   \itsp Let $\kbc$ denote the CRG induced by $\vbk\setm C$.
   In $\kb$, there can be no triangle with all edges gray, otherwise $K_{3,3}\arrows K$.  By the definition of $C$, in
   $\kbc$ there can be no triangle with all edges white
   or gray.  These edges contribute a weight of at least
   \begin{eqnarray*}
      \lefteqn{\min\{p^*,1-p^*\}\left(\binom{\vbks}{2}
      -t(\vbks,2)\right)} \\
      & & +(1-2\min\{p^*,1-p^*\})
      \left(\binom{|\vbk\setm C|}{2}-t(|\vbk\setm C|,2)\right) .
   \end{eqnarray*}
\end{itemize}

Since $p^*<0.5$, $p^*\ewks+(1-p^*)\ebks$  is at least
\begin{eqnarray*}
   \lefteqn{p^*\binom{\vwks}{2}+p^*\vwks|C|
   +p^*\left(\frac{\vbks^2-2\vbks}{4}\right)} \\
   & &  +(1-2p^*)\left(\frac{|\vbk\setm C|^2
   -2|\vbk\setm C|}{4}\right) .
\end{eqnarray*}

A lower bound on $f_K(p^*)$ gives
\begin{eqnarray}
   f_K(p^*)k^2 & = & p^*\left(\vwks+2\ewks\right)
   +(1-p^*)\left(\vbks+2\ebks\right) \nonumber \\
   & \geq & p^*\vwks+(1-p^*)\vbks+2p^*\binom{\vwks}{2} \nonumber \\
   & & +2p^*\vwks|C|
   +2p^*\left(\frac{\vbks^2-2\vbks}{4}\right) \nonumber \\
   & & +2(1-2p^*)\left(\frac{|\vbk\setm C|^2
   -2|\vbk\setm C|}{4}\right) \nonumber \\
   & \geq & (1-2p^*)|C|+p^*\vwks^2+2p^*\vwks|C| \nonumber \\
   & & +\frac{1-p^*}{2}\vbks^2-(1-2p^*)\vbks|C|
   +\frac{1-2p^*}{2}|C|^2 \nonumber \\
   & \geq & p^*\vwks^2+\frac{1-p^*}{2}\vbks^2 \nonumber \\
   & & +\frac{1-2p^*}{2}|C|\left(|C|-2\vbks
   +\frac{4p^*}{1-2p^*}\vwks\right) . \label{eq:Cterm}
\end{eqnarray}

All that remains is to verify that the expressions in (\ref{eq:Cterm}) is at most $\frac{p^*(1-p^*)}{1+p^*}k^2$.
We need to divide this into two cases.  First, assume
$\vbks\leq\frac{2p^*}{1-2p^*}\vwks$. In this case, the value of
$|C|$ that minimizes (\ref{eq:Cterm}) is
$|C|=0$,
\begin{eqnarray*}
   f_K(p^*)k^2 & \geq & p^*\vwks^2+\frac{1-p^*}{2}\vbks^2 \\
   & \geq & \left(p^*\left(\frac{1-p^*}{1+p^*}\right)^2+\frac{1-p^*}{2}
   \left(\frac{2p^*}{1+p^*}\right)^2\right)k^2 \\
   & = & \frac{p^*(1-p^*)}{1+p^*}k^2=(3-2\sqrt{2})k^2 ,
\end{eqnarray*}
because the minimum occurs at $\vbks=\frac{2p^*}{1+p^*}k$.
Second, assume
$\vbks\geq\linebreak[4]\frac{2p^*}{1-2p^*}\vwks$, i.e, $\vbks\ge 2p^*k$. In this
case, the value of $|C|$ that minimizes (\ref{eq:Cterm}) is $|C|=\vbks-\frac{2p^*}{1-2p^*}\vwks$:
\begin{eqnarray*}
   f_K(p^*)k^2
   & \geq & p^*\vwks^2+\frac{1-p^*}{2}\vbks^2
    -\frac{1-2p^*}{2}\left[\vbks
   -\frac{2p^*}{1-2p^*}\vwks\right]^2 \\
   & = & \left(p^*-\frac{2(p^*)^2}{1-2p^*}\right)\vwks^2
   +\frac{p^*}{2}\vbks^2
    +2p^*\vwks\vbks \\
   & = & p^*k^2-\frac{2(p^*)^2}{1-2p^*}\vwks^2-\frac{p^*}{2}\vbks^2
   .
\end{eqnarray*}

This expression is minimized at the endpoints of the domain of
$\vbks$.  For\linebreak[4] $\vbks=k$, we have $\frac{p^*}{2}\approx
0.207>\frac{p^*(1-p^*)}{1+p^*}=3-2\sqrt{2}\approx 0.172$. For the
other endpoint, $\vbks=2p^*k$, we have
   $$f_K(p^*)k^2  \geq  p^*k^2-\frac{2(p^*)^2}{1-2p^*}(1-2p^*)^2k^2
   -\frac{p^*}{2}(2p^*)^2k^2
   =  \left[p^*-2(p^*)^2+2(p^*)^3\right]k^2 .$$
This gives $f_K(p)\geq 15\sqrt{2}-21\approx 0.213>\frac{p^*(1-p^*)}{1+p^*}=3-2\sqrt{2}\approx 0.172$.

To summarize,
$f(p^*)\geq\frac{p^*(1-p^*)}{1+p^*}=3-2\sqrt{2}$. Therefore,
$d^*(\forb(K_{3,3}))=3-2\sqrt{2}$ and
$p^*(\forb(K_{3,3}))=\sqrt{2}-1$.

\subsection{Edit distance of $H_9$}

\subsubsection{Upper bound}
The Young tableau in Figure~\ref{fig:h9young} diagrams the values of
$(a,c)$ for which $H_9\not\arrows K(a,c)$.  To see this, we can
exhibit the following partitions of $V(G)$:
\begin{itemize}\firstspace
   \itsp \textbf{3 cliques}:
   $\left\{\{0,1,2\},\{3,4,5\},\{6,7,8\}\right\}$
   \itsp \textbf{1 coclique, 2 cliques}:
   $\left\{\{2,7\},\{8,0,1\},\{3,4,5,6\}\right\}$
   \itsp \textbf{2 cocliques, 1 clique}:
   $\left\{\{1,4,7\},\{2,5,8\},\{0,3,6\}\right\}$
   \itsp \textbf{4 cocliques}:
   $\left\{\{1,4,7\},\{0,5\},\{3,8\},\{2,6\}\right\}$.
\end{itemize}
Moreover, it is easy to see that the largest clique of $H_9$ is 4 and the largest coclique is 3.  So, $H_9\not\arrows K(0,2),K(1,1)$. Since there are only two cocliques of size 3, $H_9\not\arrows K(3,0)$.

Figure~\ref{fig:h9graph} gives the graphs of $g_{K(0,2)}(p)$,
$g_{K(3,0)}(p)$ as well as $g_{K''}(p)$ for the $K''$ defined in the
theorem.  The region they define is shaded.
\begin{center}
\begin{figure}[ht]
   \hfill
   \begin{minipage}[t]{125pt}
      \begin{center}
         \epsfig{file=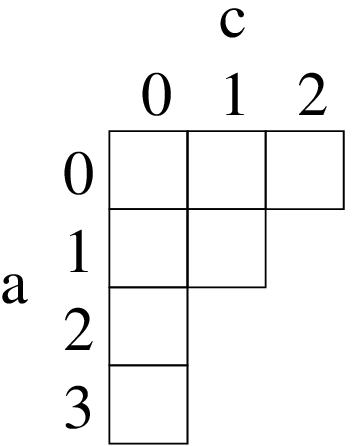,scale=0.5}
         \caption{The Young tableau of $(a,c)$ for which
         $H_9\not\arrows K(a,c)$.}
         \label{fig:h9young}
      \end{center}
   \end{minipage}
   \hfill
   \begin{minipage}[t]{250pt}
      \begin{center}
         \epsfig{file=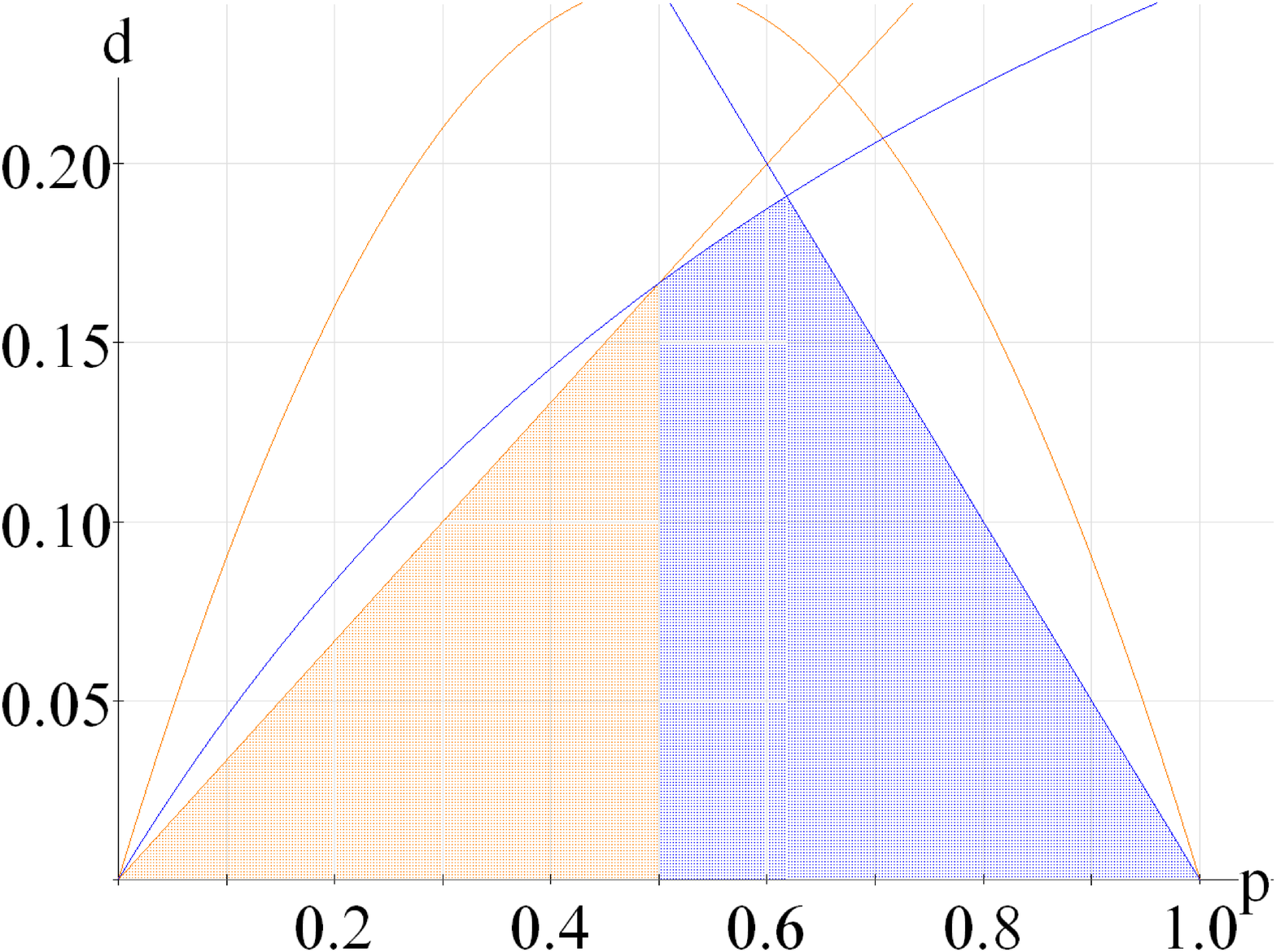,scale=0.15}
         \caption{The graphs of the $g_K(p)$ relevant to $H_9$.}
         \label{fig:h9graph}
      \end{center}
   \end{minipage}
\end{figure}
\end{center}

Recall that $K''$ satisfies $|{\rm VW}(K'')|=4$, $|{\rm
VB}(K'')|=0$, one black edge and 5 gray edges.  See
Figure~\ref{fig:k2}.
\begin{center}
\begin{figure}[ht]
   \centerline{\epsfig{file=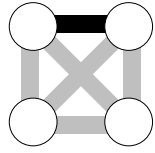,scale=0.5}}
   \caption{The colored regularity graph $K''$.} \label{fig:k2}
\end{figure}
\end{center}
The graph $H_9$ has only two cocliques of order three:
$\{1,4,7\}$ and $\{2,5,8\}$.  The vertices that remain, $\{0,3,6\}$,
form a clique.  So, any partition of the vertices of $H_9$ into
cocliques that uses both of these $3$-cocliques,
requires 5 pieces to the partition.  As a result, if there were a
colored-homomorphism from $H_9$ into $K''$, it would partition the
vertices into one coclique of order 3 and three cocliques of order 2.

Assuming such a colored-homomorphism exists, we assume, without loss
of generality, that one of the cocliques is $\{1,4,7\}$.  The vertex
$0$ is only nonadjacent to $5$.   The vertex $3$ is only nonadjacent
to $8$.  The vertex $6$ is only nonadjacent to $2$. Therefore, the
only partition that can witness the colored-homomorphism is
$\left\{\{1,4,7\},\{0,5\},\{2,6\},\{3,8\}\right\}$. Between every
pair of these cocliques is a nonedge.  So, no pair of them could be
mapped to endpoints of the black edge of $K''$. Therefore,
$H_9\not\arrows K''$.

We can use Lemma~\ref{lem:glemma} to conclude that
$g_{K'}(p)=\frac{1-p}{2}$ and $g_{K''}(p)=\frac{p}{2(1+p)}$. The
intersection is at the point
$(p,d)=\left(\frac{\sqrt{5}-1}{2},\frac{3-\sqrt{5}}{4}\right)$.

Thus, the upper bound obtained by $d^*\leq\max\limits_{p\in
[0,1]}\min\limits_{(a,c) : K(a,c)\in\K(\HH)} g_{K(a,c)}(p)$ would be
$1/5=0.2$, achieved by the intersection of
$g_{K(0,2)}(p)=\frac{1-p}{2}$ and $g_{K(3,0)}(p)=\frac{p}{3}$.  But,
as we can see by Figure~\ref{fig:h9graph}, the value
$d=\frac{3-\sqrt{5}}{4}\approx 0.191$ provides a better upper bound.

\subsection{New proof of the edit distance of $\overline{P_3+K_1}$}

Alon and Stav~\cite{AS2} computed $(p^*,d^*)$ for hereditary properties
defined by graphs on at most $4$ vertices.  This paper has also done so, as a corollary of
 our results, with the exception of $\forb\left(\overline{P_3+K_1}\right)$.
  We include a computation of the value of $(p^*,d^*)$ as an application of our
   technique and to demonstrate the versatility of Lemma~\ref{lem:KaEb}.

\subsubsection{Upper bound}
Here, we choose $K'=K(0,1)$ and $K''=K(2,0)$.  Recall that
$\overline{P_3+K_1}$ is a triangle with a pendant edge.  It is easy
to see that both $\overline{P_3+K_1}\not\arrows K'$ and
$\overline{P_3+K_1}\not\arrows K''$.  We can use
Lemma~\ref{lem:glemma} to conclude that $g_{K'}(p)=1-p$ and
$g_{K''}(p)=\frac{p}{2}$. The intersection is at the point
$(p^*,d^*)=\left(\frac{2}{3},\frac{1}{3}\right)$.  Clearly, this $p^*$ is unique because $g(p)<1/3$ for all $p\neq 2/3$.

\subsubsection{Lower bound}
Fix $p^*=\frac{2}{3}$.  Let $K$ be any CRG for which
$\overline{P_3+K_1}\not\arrows K$.  For simplicity of notation,
define $\kw$ to be the CRG induced by $\vwk$.  We will give a lower bound on $p^*\ewks+(1-p^*)\ebks$.
\begin{itemize}\firstspace
   \itsp In the bipartite CRG induced by $(\vwk,\vbk)$, no edges can be gray,
   otherwise $\overline{P_3+K_1}\arrows K$.  This contributes
   $\min\{p^*,1-p^*\}\vwks\vbks$.
   \itsp In the CRG induced by $\vbk$, no edges can be gray,
   otherwise $\overline{P_3+K_1}\linebreak\arrows K$.   This contributes
   $\min\{p^*,1-p^*\}\binom{\vbks}{2}$.
   \itsp In $\kw$, consider any subset of $3$ vertices.  If
   there is neither a white edge nor a pair of black edges, then it is
   possible to map the vertices of $\overline{P_3+K_1}$
   into those three vertices so that each vertex of the triangle is
   mapped to a different vertex and the
   pendant vertex is in the vertex incident to the two gray
   edges.  This contributes $p^*\ew(\kw)+(1-p^*)\eb(\kw)$.
\end{itemize}

To summarize, using $p^*=2/3$ and  Lemma~\ref{lem:KaEb}, we have
\begin{eqnarray*}
   \lefteqn{p^*\ewks+(1-p^*)\ebks} \\ & \geq &
   \frac{1}{3}\left(\vwks\vbks+\binom{\vbks}{2}\right)
   +\frac{1}{3}\left(2|\ew(\kw)|+|\eb(\kw)|\right) \\
   & \geq &
   \frac{1}{3}\left(\vwks\vbks+\binom{\vbks}{2}\right)
   +\frac{1}{3}\left\lceil\frac{\vwks}{2}(\vwks-2)
   \right\rceil \\
   & = & \frac{1}{3}\left(\vwks\vbks+\binom{\vbks}{2}
   +\frac{\vwks^2-2\vwks}{2}\right) .
\end{eqnarray*}
Now we give  a lower bound on $f_K(p^*)$:
\begin{eqnarray*}
   f_K(p^*)k^2 & = & p^*\left(\vwks+2\ewks\right)
   +(1-p^*)\left(\vbks+2\ebks\right) \\
   & \geq & \frac{2}{3}\vwks+\frac{1}{3}\vbks
   +\frac{2}{3}\left(\vwks\vbks\right. \\
   & & +\left.\frac{\vbks^2-\vbks}{2}
   +\frac{\vwks^2-2\vwks}{2}\right) \\
   & = & \frac{1}{3}\left(\vwks^2+2\vwks\vbks+\vbks^2\right)
   =\frac{1}{3}k^2 .
\end{eqnarray*}

So, comparing with the upper bound, $d^*(\forb(\overline{P_3+K_1}))=1/3$ and since
 $g(p)\leq\min\{g_{K'}(p),$ $g_{K''}(p)\}<1/3$ for all $p\neq 2/3$, it is also the case that
$p^*(\forb(\overline{P_3+K_1}))=2/3$. \hfill $\Box$

\section{Conclusions}

\subsection{Observations on functions $f$ and $g$}
The function $f_K(p)$ is invariant under equipartitions of $V(K)$.
To see this, let $\tilde{K}$ be formed by partitioning each vertex
of $K$ into $c$ pieces with the colors of vertices and edges be the
natural coloring inherited from $K$.  As a result,
$|\vw(\tilde{K})|=c\vwks$, $|\vb(\tilde{K})|=c\vbks$,
$|\ew(\tilde{K})|=c^2\ewks+\binom{c}{2}\vwks$ and
$|\eb(\tilde{K})|=c^2\ebks+\binom{c}{2}\vbks$.  Thus, $$
f_{\tilde{K}}(p)  =  \frac{1}{c^2k^2}
   \left[p(|\vb(\tilde{K})|+2|\ew(\tilde{K})|)
   +(1-p)(|\vw(\tilde{K})|+2|\eb(\tilde{K})|)\right] \\
    =  f_K(p) .
$$

The same is true for $g_K(p)$ and $g_{\tilde{K}}(p)$.  Any feasible
solution, $\uu$ of the quadratic program that defines $g_K(p)$ can be
made into a feasible solution, $\tilde{\uu}$ of the quadratic program
that defines $g_{\tilde{K}}(p)$ by arbitrarily distributing the
weight of one vertex in $K$ to the vertices in $\tilde{K}$ to which
it corresponds.  It can be seen that, if $\M_K(p)$ is the matrix
corresponding to $K$ and $\M_{\tilde{K}}(p)$ is the matrix
corresponding to $\tilde{K}$, then
$\uu^T\M_K(p)\uu=\tilde{\uu}^T\M_{\tilde{K}}(p)\tilde{\uu}$.

The function $g$ is more flexible, however.  It is not only
invariant under equipartitions but it is invariant under arbitrary
partitions. To see this, construct an equivalence relation on the
vertices of a CRG, $K$, in which vertices $u$ and $v$ are equivalent
if  $u$, $v$ and $\{u,v\}$ are all the same color and, for all $w\in
V(K)\setm\{u,v\}$, $\{u,w\}$ and $\{v,w\}$ are the same color as
each other.

If $K$ is a CRG and $K_0$ is the CRG induced by the equivalence
relation on $K$, then $g_K(p)=g_{K_0}(p)$.  Therefore, in the
computation of $g(p)$, one may ignore CRGs which have nontrivial
equivalence classes.

\subsection{Open questions}
\begin{itemize}\firstspace\firstspace
   \itsp Investigating Proposition~\ref{prop:UB}, is there a more
   convenient expression for the upper bound based only on the Young
   diagram (see Figures~\ref{fig:k33young} and~\ref{fig:h9young}) of
   the set of CRGs $\{K(a,c) : H\not\arrows K(a,c), \forall
   H\in\F(\HH)\}$?

   \itsp To compute the edit distance is hard, we do not have sharp
   result even for $\forb(K_{m,n})$, where $K_{m,n}$ is an arbitrary
   complete bipartite graph.

   \itsp The precise value for $d^*(H_9)$, where $H_9$ is defined
   in Theorem~\ref{thm:h9}, is unknown, but we conjecture that the
   upper bound is correct and we further conjecture that
   $p^*(\forb(H_9))=(\sqrt{5}-1)/2$.

   \itsp Every hereditary property $\HH$ can be expressed as
   $\bigcap_{H\in\F(\HH)}\forb(H)$ for some family of graphs
   $\F(\HH)$.  In computing edit distance, it may be that some
   members of $H\in\F(\HH)$ are unnecessary, even if they are
   necessary to define the family.  I.e, for hereditary property
   $\HH$, what are the maximal properties $\HH'\supseteq\HH$ such
   that $d^*(\HH')=d^*(\HH)$?

   For example, the strong perfect graph theorem~\cite{CRST} states
   that perfect graphs are characterized by $\PP=\bigcap_{k\geq
   2}\left(\forb(C_{2k+1})\cap\forb(\overline{C_{2k+1}})\right)$.
   But, it is not difficult to use Theorem~\ref{thm:cnaught} to show that
   $d^*(\PP)=d^*(\forb(C_5))=1/4$, as observed by Alon and Stav~\cite{AS1}.

   \itsp Our proofs of the lower bounds for
   $d^*(\forb(H))$ for $H=K_a+E_b$ or $H=K_{3,3}$ are
   cumbersome and cannot assume that the total number of
   vertices in each of the forbidden CRGs is bounded by any function of $H$.  Is there a better
   way to compute the lower bound?  Is there a function of $H$ so
   that we need only to consider $g_K(p)$ for $K$ whose order is
   bounded by said function?

   \itsp Finally, the edit distance is unknown for most hereditary
   properties.  So-called unit disk graphs (UDGs), see~\cite{CCJ},
   define a hereditary property but the family $\F$ is not known.
   It is easy to see that $K_{1,7}$ cannot occur as an induced
   subgraph in a UDG.  (For some definitions of the unit disk graph,
   $K_{1,6}$ is forbidden also.)  We believe that a small family of
   such forbidden induced subgraphs will be enough to determine the
   edit distance from the family of UDGs.

       For any graph $H$, both $p^*(\forb(H))$ and $d^*(\forb(H))$
       can be considered invariants of graph $H$.  Being able to
       compute these invariants even for some given fixed graph
       seems to be quite difficult in general.
\end{itemize}

\subsection{Thanks}
We thank Maria Axenovich for valuable conversations. We also thank
Noga Alon for reading a draft of the paper and making several useful
comments. These include the observation that the existence of the
limit that defines $d^*(\HH)$ results from a simple monotonicity
argument.  We appreciate the careful reading of an anonymous referee, which lead to improvements in the manuscript.

\end{document}